\documentclass[9pt,shortpaper,twoside,web]{IEEEtran}
\usepackage{cite}

\ifCLASSINFOpdf
  \usepackage[pdftex]{graphicx}
  \usepackage{epstopdf}
  \usepackage{subfigure}
  \graphicspath{{./}{./figures/}}
\else
\fi
\usepackage{xcolor}

\usepackage{amsmath}
\usepackage{amssymb}

\def\R{\mathbb{R}}

\newtheorem{nnassumption}{\bf Assumption}

\newtheorem{nntheorem}{\bf Theorem}
\newenvironment{theorem}{\begin{nntheorem}\it}{\end{nntheorem}}
\newtheorem{nncorollary}{\bf Corollary}

\newtheorem{nndefinition}{\bf Definition}

\newtheorem{nnproposition}{\bf Proposition}

\newtheorem{nnproblem}{\bf Problem}

\newtheorem{nnlemma}{\bf Lemma}
\newenvironment{lemma}{\begin{nnlemma}\it}{\end{nnlemma}}
\newtheorem{nnremark}{\bf Remark}
\newenvironment{remark}{\begin{nnremark} \rm }{\hfill \hspace*{1pt}\hfill $\circ$\end{nnremark}}
\newtheorem{nnexample}{\bf Example}

\newenvironment{proof}{{\bf Proof.}}{\hfill \hspace*{1pt}\hfill $\Box$}

\allowdisplaybreaks

\begin{document}
%
\title{Local output feedback stabilization of a Reaction-Diffusion equation with saturated actuation}
%
%
%

\author{Hugo~Lhachemi, Christophe Prieur
\thanks{Hugo Lhachemi is with Universit{\'e} Paris-Saclay, CNRS, CentraleSup{\'e}lec, Laboratoire des signaux et syst{\`e}mes, 91190, Gif-sur-Yvette, France (e-mail: hugo.lhachemi@centralesupelec.fr).  The work of H. Lhachemi has been partially supported by ANR PIA funding: ANR-20-IDEES-0002.
\newline\indent Christophe Prieur is with Universit{\'e} Grenoble Alpes, CNRS, Grenoble-INP, GIPSA-lab, F-38000, Grenoble, France (e-mail: christophe.prieur@gipsa-lab.fr). The work of C. Prieur has been partially supported by MIAI@Grenoble Alpes (ANR-19-P3IA-0003).
}
}

%
%

\markboth{}%
{Lhachemi \MakeLowercase{\textit{et al.}}}
%



\maketitle

\begin{abstract}
This paper is concerned with the output feedback stabilization of a reaction-diffusion equation by means of bounded control inputs in the presence of saturations. Using a finite-dimensional controller composed of an observer coupled with a finite-dimensional state-feedback, we derive a set of conditions ensuring the stability of the closed-loop plant while estimating the associated domain of attraction in the presence of saturations. This set of conditions is shown to be always feasible for an order of the observer selected large enough. The stability analysis relies on Lyapunov functionals along with a generalized sector condition classically used to study the stability of linear finite-dimensional plants in the presence of saturations.
\end{abstract}

\begin{IEEEkeywords}
Saturated control, domain of attraction, reaction-diffusion equation, output feedback.
\end{IEEEkeywords}

%
\IEEEpeerreviewmaketitle

\section{Introduction}\label{sec: Introduction}

Saturation mechanisms are ubiquitous in practical applications and impose severe constraints on control design~\cite{bernstein1995chronological}. Even in the favorable case of finite-dimensional linear time-invariant systems, saturations are well-known for their capability to introduce instabilities mechanisms as well as the loss of global stability properties~\cite{campo1990robust}. We refer to~\cite{tarbouriech2011stability,zaccarian2011modern} for comprehensive introductions to the topic of feedback control in the presence of saturations.

The impact of input saturations on the application of different control strategies for infinite-dimensional plants has been investigated in a number of papers. Among the very first contributions in this research direction, one can find the seminal works~\cite{slemrod1989feedback,lasiecka2003strong}. The saturation mechanisms considered therein are expressed in terms of control input functions evaluated in the norm of an abstract functional space. However, saturations encountered in practical applications generally consist in pointwise saturation mechanisms. Such pointwise saturation scenarios have been the topic of a number of recent papers. The stabilization of different types of PDEs, including e.g. wave and Korteweg-de Vries equations, under cone-bounded feedback (which in particular include saturations) have been reported in~\cite{prieur2016wave,marx2017cone,marx2017global} based on Lyapunov stability analysis methods. The feedback stabilization of reaction-diffusion equations with control input constraints was reported in~\cite{dubljevic2006predictive} by means of model predictive control and in~\cite{el2003analysis} by means of singular perturbation techniques. The use of spectral reduction methods for the state-feedback of a reaction-diffusion equation was recently reported in~\cite{mironchenko2020local} with explicit estimation of the domain of attraction based on linear (LMIs) and bilinear (BMIs) matrix inequalities.

This paper considers the problem of local output feedback stabilization of a reaction-diffusion equation by means of saturated control inputs. In this setting, the control inputs apply in the domain by means of a bounded operator while the observation can take the form of either a bounded or an unbounded measurement operator. The latter scenario covers the cases of Dirichlet and Neumann measurements. The adopted approach relies on spectral-reduction methods which have been extensively used in a great variety of contexts for the control of parabolic PDEs~\cite{russell1978controllability,coron2004global,lhachemi2020feedback,lhachemi2020pi,lhachemi2021indomain,mironchenko2020local}. This is particularly because the state-feedback setting for the control of parabolic PDEs is well-known to allow the design of the control strategy on a finite dimensional truncated model capturing the unstable dynamics while ensuring the preservation of the stability properties of the residual infinite-dimension dynamics~\cite{russell1978controllability}. With such a state-feedback control strategy, the presence of saturated actuators induce that the domain of attraction is only limited by the domain of attraction of the finite-dimensional truncated model while the initial conditions of the residual modes are unconstrained. We refer to~\cite[Proposition~1]{mironchenko2020local} for a precise statement of this result. In the case of an output feedback strategy such as the ones described in this paper, such a decoupling is not possible in general. This is in particular the reason why the subsets of the domain of attraction of the closed-loop plant that we derive in this paper impose constraints on the initial conditions of all the modes of the reaction-diffusion plant simultaneously.

The employed control strategy takes the form of a finite-dimensional controller composed of a finite-dimensional observer coupled with a finite-dimensional state-feedback~\cite{sakawa1983feedback,balas1988finite,curtain1982finite}. We take advantage of the control architecture initially reported in~\cite{sakawa1983feedback} augmented with the LMI-based procedure introduced in~\cite{katz2020constructive}. This approach has been extended in~\cite{lhachemi2020finite} to the case of boundary control and a either Dirichlet or Neumann measurement, as well as in~\cite{lhachemi2020finitePI} for PI regulation control. In this paper, the presence of the input saturation is handled in the stability analysis by invoking a generalized sector condition~\cite[Lem.~1.6]{tarbouriech2011stability}. While this sector condition has emerged as the predominant technique to study the stability of saturated linear finite-dimensional plants, this paper reports its very first use in the context of output feedback stabilization of infinite-dimensional systems. Combining this generalized sector condition with a Lyapunov-based analysis procedure adapted from~\cite{lhachemi2020finite}, we derive different sets of explicit constraints ensuring the exponential stability of the system trajectories for all initial conditions belonging to a subset of the region of attraction. It is worth noting that, due to the use of the generalized sector condition, certain of these initial conditions may induce the actual saturation of the command during the transient. This is in sharp contrast with a number of works in which very conservative subsets of the domain of attraction of the PDE system are obtained by restraining the norm of the initial condition so that the saturation mechanism is never active, see e.g. \cite{kang2017boundary} in a state feedback context. The obtained subset of the domain of attraction is characterized by the decision variables of the abovementionned constraints. In this process, we show that the order of the controller can always be selected large enough so that the abovementionned constraints can be satified. 

The rest of this paper is organized as follows. After introducing some notations and general properties of the Sturm-Liouville operators, the control design problem is introduced in Section~\ref{sec: preliminaries} for a bounded measurement operator. The controller architecture, the stability analysis, and the subsequent stability properties are presented in Section~\ref{sec: control design}. The theoretical results are then illustrated based on numerical computations in Section~\ref{sec: numerical example}. The extension of the stability results to the cases of Dirichlet and Neumann boundary measurement are reported in Section~\ref{sec: extension}. Finally, concluding remarks are formulated in Section~\ref{sec: conclusion}.

\section{Notation, properties, and problem description}\label{sec: preliminaries}

\subsection{Notation}

Spaces $\R^n$ are endowed with the Euclidean norm denoted by $\Vert\cdot\Vert$. The associated induced norms of matrices are also denoted by $\Vert\cdot\Vert$. For any $x,y\in\R^n$, $x \leq y$ means that each component of $x$ is less than or equal to the corresponding component of $y$. For any $x \in \R^n$ we denote by $\vert x \vert$ the vector of $\R^n$ obtained by replacing each component of $x$ by its absolute value. Given two vectors $X$ and $Y$, $ \mathrm{col} (X,Y)$ denotes the vector $[X^\top,Y^\top]^\top$. $L^2(0,1)$ stands for the space of square integrable functions on $(0,1)$ and is endowed with the inner product $\langle f , g \rangle = \int_0^1 f(x) g(x) \,\mathrm{d}x$ with associated norm denoted by $\Vert \cdot \Vert_{L^2}$. For an integer $m \geq 1$, the $m$-order Sobolev space is denoted by $H^m(0,1)$ and is endowed with its usual norm denoted by $\Vert \cdot \Vert_{H^m}$. For a symmetric matrix $P \in\R^{n \times n}$, $P \succeq 0$ (resp. $P \succ 0$) means that $P$ is positive semi-definite (resp. positive definite) while $\lambda_M(P)$ (resp. $\lambda_m(P)$) denotes its maximal (resp. minimal) eigenvalue.

Let $(\phi_n)_{n \geq 1}$ be a Hilbert basis of $L^2(0,1)$. For any integers $1 \leq N < M$, we define the operators of projection $\pi_{N} : L^2(0,1) \rightarrow \R^{N}$ and $\pi_{N,M} : L^2(0,1) \rightarrow \R^{M-N}$ by setting $\pi_{N} f = \begin{bmatrix} \left< f , \phi_1 \right> & \ldots & \left< f , \phi_{N} \right> \end{bmatrix}^\top$ and $\pi_{N,M} f = \begin{bmatrix} \left< f , \phi_{N+1} \right> & \ldots & \left< f , \phi_{M} \right> \end{bmatrix}^\top$. We also define $\mathcal{R}_{N} : L^2(0,1) \rightarrow L^2(0,1)$ by $\mathcal{R}_{N} f = f - \sum_{n=1}^{N} \left< f, \phi_n \right> \phi_n = \sum_{n \geq N+1} \left< f, \phi_n \right> \phi_n$.

\subsection{Properties of Sturm-Liouville operators}\label{subsec: Properties of Sturm-Liouville operators}

Let $\theta_1,\theta_2\in[0,\pi/2]$, $p \in \mathcal{C}^1([0,1])$, and $q \in \mathcal{C}^0([0,1])$ with $p > 0$ and $q \geq 0$. Consider the Sturm-Liouville operator $\mathcal{A} : D(\mathcal{A}) \subset L^2(0,1) \rightarrow L^2(0,1)$ defined by $\mathcal{A}f = - (pf')' + q f$ on the domain $D(\mathcal{A}) = \{ f \in H^2(0,1) \,:\, \cos(\theta_1) f(0) - \sin(\theta_1) f'(0) = 0 , \, \cos(\theta_2) f(1) + \sin(\theta_2) f'(1) = 0 \}$. The eigenvalues $\lambda_n$, $n \geq 1$, of $\mathcal{A}$ are simple, non-negative, and form an increasing sequence with $\lambda_n \rightarrow + \infty$ as $n \rightarrow + \infty$. Moreover, the associated unit eigenvectors $\phi_n \in L^2(0,1)$ form a Hilbert basis. We also have $D(\mathcal{A}) = \{ f \in L^2(0,1) \,:\, \sum_{n\geq 1} \vert \lambda_n \vert ^2 \vert \left< f , \phi_n \right> \vert^2 < +\infty \}$ and $\mathcal{A}f = \sum_{n \geq 1} \lambda_n \left< f , \phi_n \right> \phi_n$. For $f \in D(\mathcal{A})$ we consider $\mathcal{A}^{1/2}f = \sum_{n \geq 1} \lambda_n^{1/2} \left< f , \phi_n \right> \phi_n$.

Let $p_*,p^*,q^* \in \R$ be such that $0 < p_* \leq p(x) \leq p^*$ and $0 \leq q(x) \leq q^*$ for all $x \in [0,1]$, then it holds~\cite{orlov2017general}:
\begin{equation}\label{eq: estimation lambda_n}
0 \leq \pi^2 (n-1)^2 p_* \leq \lambda_n \leq \pi^2 n^2 p^* + q^*
\end{equation}
for all $n \geq 1$. Assuming further that $q > 0$, an integration by parts and the continuous embedding $H^1(0,1) \subset L^\infty(0,1)$ (see, e.g., \cite[Thm.~8.8]{brezis2011functional}) show the existence of constants $C_1,C_2 > 0$ such that we have
\begin{align}
C_1 \Vert f \Vert_{H^1}^2 \leq 
\sum_{n \geq 1} \lambda_n \left< f , \phi_n \right>^2
= \left< \mathcal{A}f , f \right>
\leq C_2 \Vert f \Vert_{H^1}^2 \label{eq: inner product Af and f}
\end{align}
for all $f \in D(\mathcal{A})$. In that case, we easily deduce that $f(0) = \sum_{n \geq 1} \left< f , \phi_n \right> \phi_n(0)$ and $f'(0) = \sum_{n \geq 1} \left< f , \phi_n \right> \phi_n'(0)$ hold for all $f \in D(\mathcal{A})$. Finally, if we further assume that $p \in \mathcal{C}^2([0,1])$, we obtain that $\phi_n(0) = O(1)$ and $\phi_n'(0) = O(\sqrt{\lambda_n})$ as $n \rightarrow +\infty$~\cite{orlov2017general}.

\subsection{Problem description}

We consider the reaction-diffusion equation with Robin boundary conditions described by
\begin{subequations}\label{eq: distributed measurement - RD system}
\begin{align}
& z_t(t,x) = \left( p(x) z_x(t,x) \right)_x - \tilde{q}(x) z(t,x) \label{eq: distributed measurement - RD system - 1} \\
& \phantom{z_t(t,x) = }\; + \sum_{k=1}^m b_k(x) u_{\mathrm{sat},k}(t) \nonumber \\
&\cos(\theta_1) z(t,0) - \sin(\theta_1) z_x(t,0) = 0 \label{eq: distributed - RD system - 2} \\
&\cos(\theta_2) z(t,1) + \sin(\theta_2) z_x(t,1) = 0 \label{eq: distributed - RD system - 3} \\
& z(0,x) = z_0(x) . \label{eq: distributed - RD system - 4}
\end{align}
\end{subequations}
Here we have $\theta_1,\theta_2 \in [0,\pi/2]$, $p \in \mathcal{C}^1([0,1])$ with $p > 0$, and $\tilde{q} \in \mathcal{C}^0([0,1])$. Moreover $b_k \in L^2(0,1)$ represents the manner the scalar control input $u_{\mathrm{sat},k}(t) \in\R$ acts on the system, $z_0 \in L^2(0,1)$ is the initial condition, and $z(t,\cdot) \in L^2(0,1)$ is the state of the reaction-diffusion PDE. We introduce without loss of generality\footnote{For example one can select $q_c = - \min\tilde{q}$ and $q(x) = \tilde{q}(x)+q_c$.} a function $q \in \mathcal{C}^0([0,1])$ and a constant $q_c \in \R$ such that 
\begin{equation}\label{eq: decomposition of the reaction term}
\tilde{q} = q - q_c , \quad q \geq 0 .
\end{equation}
Hence (\ref{eq: distributed measurement - RD system}) can be rewritten in abstract form as 
\begin{subequations}\label{eq: distributed measurement - RD system - abstract}
\begin{align}
& z_t(t,\cdot) = \left\{ -\mathcal{A} + q_c \mathrm{Id}_{L^2} \right\} z(t,\cdot) + \sum_{k=1}^m b_k u_{\mathrm{sat},k}(t) \\
& z(0,\cdot) = z_0  
\end{align}
\end{subequations}
where $\mathcal{A}$ is the Sturm-Liouville operator defined in Subsection~\ref{subsec: Properties of Sturm-Liouville operators}. Reaction-diffusion PDEs with in-domain actuation described by (\ref{eq: distributed measurement - RD system}) are encountered in practical applications such as thermonuclear fusion with Tokamaks~\cite{mavkov2017distributed}, stabilization of fronts in chemical reactors~\cite{smagina2002stabilization}, surface decontamatination~\cite{wang2021delay}, and population dynamics~\cite{wang2020regional}.

We start this study by assuming that the measurement takes the form of a bounded observation operator, i.e., 
\begin{equation}\label{eq: bounded measurement operator}
y(t) = \int_0^1 c(x) z(t,x) \,\mathrm{d}x 
\end{equation}
for some $c \in L^2(0,1)$. This type of measurement covers the case of in-domain sensors measuring an averaged value of the state in a spatial neighborhood. It is worth noting that bounded observation operators are generally easier to handle from a control design perspective. The extension of the control design procedure from bounded to unbounded observation operators is reported in Section~\ref{sec: extension}.

The application of the control input $u_{\mathrm{sat},k}(t)$ is assumed to be subject to saturations. More specifically, we define for an arbitrary $l > 0$ the saturation function $\mathrm{sat}_l : \R \rightarrow \R$ as
\begin{equation*}
\mathrm{sat}_l(v) = 
\left\{\begin{split} 
v & \qquad\mathrm{if}\; \vert v \vert \leq l ; \\ 
l \frac{v}{\vert v \vert} & \qquad\mathrm{if}\; \vert v \vert \geq l .
\end{split}\right. 
\end{equation*}
For a given $\ell = \begin{bmatrix} l_1 & l_2 & \ldots & l_m \end{bmatrix}^\top\in(\R_{>0})^m$, we define $\mathrm{sat}_\ell : \R^m \rightarrow \R^m$ as
\begin{equation*}
\mathrm{sat}_\ell(v) = \begin{bmatrix} \mathrm{sat}_{l_1}(v_1) & \mathrm{sat}_{l_2}(v_2) & \ldots & \mathrm{sat}_{l_m}(v_m) \end{bmatrix}^\top .
\end{equation*}
Hence, we assume that the actually applied control $u_{\mathrm{sat},k}(t)$ is linked to the nominally designed control input $u_{k}(t)$ by
\begin{equation*}
u_{\mathrm{sat},k}(t) = \mathrm{sat}_{l_k}(u_k(t)) 
\end{equation*}
for all $1 \leq k \leq m$. Introducing 
\begin{align*}
u(t) & = \begin{bmatrix} u_1(t) & u_2(t) & \ldots & u_m(t) \end{bmatrix}^\top \in\R^m , \\
u_\mathrm{sat}(t) & = \begin{bmatrix} u_{\mathrm{sat},1}(t) & u_{\mathrm{sat},2}(t) & \ldots & u_{\mathrm{sat},m}(t) \end{bmatrix}^\top \in\R^m
\end{align*}
the latter identity reads in compact form
\begin{equation*}
u_\mathrm{sat}(t) = \mathrm{sat}_\ell(u(t)) . 
\end{equation*}

The objective is to design a finite-dimensional output feedback controller in order to achieve the local stabilization of (\ref{eq: distributed measurement - RD system}) with measurement (\ref{eq: bounded measurement operator}) while estimating the associated domain of attraction in the presence of the saturating control inputs. To do so, it is relevant to introduce the deadzone nonlinearity $\phi_\ell : \R^m \rightarrow \R^m$ defined for any $v\in\R^m$ by
\begin{equation}\label{eq: deadzone nonlinearity}
\phi_\ell(v) = \mathrm{sat}_\ell(v)-v .
\end{equation}
This representation is mainly motivated by the fact that this deadzone nonlinearity satisfies the following generalized sector condition borrowed from~\cite[Lem.~1.6]{tarbouriech2011stability}.

\begin{lemma}\label{eq: generalized sector condition}
Let $\ell \in(\R_{>0})^m$ be given. For any $v,\omega \in\R^m$ such that $\vert v-\omega \vert \leq \ell$ and any diagonal positive definite matrix $\mathbf{T} \in \R^{m \times m}$ we have $\phi_\ell(v)^\top \mathbf{T} ( \phi_\ell(v) + \omega ) \leq 0$.
\end{lemma}

\begin{remark}
The problem of saturated feedback control of the reaction-diffusion plant (\ref{eq: distributed measurement - RD system}) and estimation of the associated region of attraction was studied in~\cite{mironchenko2020local} in the case of a state-feedback (the studied setting was focused on Dirichlet boundary conditions but extends in a straightforward manner to Robin boundary conditions). We focus here in this paper onto the case of an output feedback. 
\end{remark}

\section{Control architecture and stability results}\label{sec: control design}

\subsection{Control architecture}\label{subsec: control architecture}

Define the coefficients of projection $z_n(t) = \left< z(t,\cdot) , \phi_n \right>$, $b_{n,k} = \left< b_k , \phi_n \right>$, and $c_{n} = \left< c , \phi_n \right>$. Then the projection of the system trajectories (\ref{eq: distributed measurement - RD system - abstract}) and the output equation (\ref{eq: bounded measurement operator}) into the Hilbert basis $\{ \phi_n \,:\, n \geq 1\}$ gives the following representation:
\begin{subequations}\label{eq: distributed measurement - projection}
\begin{align}
\dot{z}_n(t) & = (-\lambda_n+q_c) z_n(t) + \sum_{k=1}^m b_{n,k} u_{\mathrm{sat},k}(t) \\
y(t) & = \sum_{n \geq 1} c_n z_n(t)  \label{eq: distributed measurement - projection - 2}
\end{align}
\end{subequations}
We consider the feedback law taking the form of a finite-dimensional state-feedback coupled with a finite-dimensional observer~\cite{sakawa1983feedback,katz2020constructive,lhachemi2020finite}. More precisely, let $\delta > 0$ and $N_0 \geq 1$ be such that $-\lambda_n + q_c < - \delta$ for all $n \geq N_0 + 1$. For a given integer $N \geq N_0 +1$ to be selected later, the controller architecture takes the form:
\begin{subequations}\label{eq: control law}
\begin{align}
\dot{\hat{z}}_n(t) & = (-\lambda_n+q_c) \hat{z}_n(t) + \sum_{k=1}^m b_{n,k} u_{\mathrm{sat},k}(t) \\
& \phantom{=}\; - L_n \left\{ \sum_{k=1}^N c_k \hat{z}_k(t) - y(t) \right\}  , \quad 1 \leq n \leq N \nonumber \\
u_k(t) & = \sum_{l=1}^{N_0} K_{k,l} \hat{z}_l(t) , \quad 1 \leq k \leq m
\end{align}
\end{subequations}
with $L_n,K_{k,l}\in\R$ where $L_n = 0$ for $N_0 +1 \leq n \leq N$. See~\cite{sakawa1983feedback} for an early occurrence of such a control architecture.

\begin{remark}
Due to the measurement (\ref{eq: distributed measurement - projection - 2}), the dynamical controller (\ref{eq: control law}) receives in input all the infinite number of modes of the reaction-diffusion plant (\ref{eq: distributed measurement - RD system}). Therefore, the resulting closed-loop system cannot be recasted into the setting of~\cite[Sec.~III.E]{mironchenko2020local} except in the very particular case where $c_n = 0$ for all $n \geq N+1$.
\end{remark}

We define the error signals $e_n(t) = z_n(t) - \hat{z}_n(t)$. Introducing the vectors and matrices defined by $\hat{Z}^{N_0} = \begin{bmatrix} \hat{z}_1 & \ldots & \hat{z}_{N_0} \end{bmatrix}^\top$, $\hat{Z}^{N-N_0} = \begin{bmatrix} \hat{z}_{N_0 + 1} & \ldots & \hat{z}_{N} \end{bmatrix}^\top$, $E^{N_0} = \begin{bmatrix} e_1 & \ldots & e_{N_0} \end{bmatrix}^\top$, $E^{N-N_0} = \begin{bmatrix} e_{N_0 + 1} & \ldots & e_{N} \end{bmatrix}^\top$, $A_0 = \mathrm{diag}(-\lambda_1+q_c,\ldots,-\lambda_{N_0}+q_c)$, $A_1 = \mathrm{diag}(-\lambda_{N_0 + 1}+q_c,\ldots,-\lambda_{N}+q_c)$, $B_0 = (b_{n,k})_{1 \leq n \leq N_0 , 1 \leq k \leq m}$, $B_1 = (b_{n,k})_{N_0 + 1 \leq n \leq N , 1 \leq k \leq m}$, $C_0 = \begin{bmatrix} c_1 & \ldots & c_{N_0} \end{bmatrix}$, $C_1 = \begin{bmatrix} c_{N_0 + 1} & \ldots & c_{N} \end{bmatrix}$, $L = \begin{bmatrix} L_1 & \ldots & L_{N_0} \end{bmatrix}^\top$, and $K = (K_{k,l})_{1 \leq k \leq m , 1 \leq l \leq N_0}$, we infer that
\begin{align*}
\dot{\hat{Z}}^{N_0} & = A_0 \hat{Z}^{N_0} + B_0 u_\mathrm{sat} + L C_0 E^{N_0} + L C_1 E^{N-N_0} + L \zeta \\
\dot{E}^{N_0} & = (A_0 - L C_0) E^{N_0} - L C_1 E^{N-N_0} - L \zeta \\
\dot{\hat{Z}}^{N-N_0} & = A_1 \hat{Z}^{N-N_0} + B_1 u_\mathrm{sat} \\
\dot{E}^{N-N_0} & = A_1 E^{N-N_0}
\end{align*}
with
\begin{equation*}
u = K \hat{Z}^{N_0}
\end{equation*}
and where $\zeta(t) = \sum_{n \geq N+1} c_n z_n(t)$. Using the deadzone nonlinearity (\ref{eq: deadzone nonlinearity}) we obtain that
\begin{align*}
\dot{\hat{Z}}^{N_0} & = (A_0 + B_0 K) \hat{Z}^{N_0} + L C_0 E^{N_0} + L C_1 E^{N-N_0} \\
& \phantom{=}\; + L \zeta + B_0 \phi_\ell(K \hat{Z}^{N_0}) \\
\dot{E}^{N_0} & = (A_0 - L C_0) E^{N_0} - L C_1 E^{N-N_0} - L \zeta \\
\dot{\hat{Z}}^{N-N_0} & = A_1 \hat{Z}^{N-N_0} + B_1 K \hat{Z}^{N_0} + B_1 \phi_\ell(K \hat{Z}^{N_0}) \\
\dot{E}^{N-N_0} & = A_1 E^{N-N_0} .
\end{align*}
Introducing the state-vector
\begin{equation}\label{eq: state-vector closed-loop truncated model}
X = \mathrm{col} ( \hat{Z}^{N_0} , E^{N_0} , \hat{Z}^{N-N_0} , E^{N-N_0} )
\end{equation}
as well as the matrices 
\begin{equation}\label{eq: def matrix F}
F = \begin{bmatrix}
A_0 + B_0 K & L C_0 & 0 & LC_1 \\ 0 & A_0-LC_0 & 0 & -L C_1 \\ B_1 K & 0 & A_1 & 0 \\ 0 & 0 & 0 & A_1
\end{bmatrix} ,
\end{equation}
$\mathcal{L} = \mathrm{col} ( L , -L , 0 , 0 )$ and $\mathcal{L}_\phi = \mathrm{col} ( B_0 , 0 , B_1 , 0 )$, we obtain that
\begin{equation}\label{eq: dynamics truncated model}
\dot{X} = F X + \mathcal{L} \zeta + \mathcal{L}_\phi \phi_\ell(K\hat{Z}^{N_0}) .
\end{equation}
Defining $E = \begin{bmatrix} I & 0 & 0 & 0\end{bmatrix}$ and $\tilde{K} = \begin{bmatrix} K & 0 & 0 & 0 \end{bmatrix}$, we also have $\hat{Z}^{N_0} = E X$ and $u = \tilde{K} X$.

\subsection{Main results}

\subsubsection{Stabilization in $L^2$ norm}

Our first main result is stated below.

\begin{theorem}\label{thm: main result 1}
Let $\theta_1,\theta_2\in[0,\pi/2]$, $p \in \mathcal{C}^1([0,1])$ with $p > 0$, $\tilde{q} \in \mathcal{C}^0([0,1])$, and $\ell\in(\R_{>0})^m$. Let $q \in \mathcal{C}^0([0,1])$ and $q_c \in\R$ be such that (\ref{eq: decomposition of the reaction term}) holds. Let $c \in L^2(0,1)$ and $b_k \in L^2(0,1)$ for $1 \leq k \leq m$. Consider the reaction-diffusion system described by (\ref{eq: distributed measurement - RD system}) with measured output (\ref{eq: bounded measurement operator}). Let $N_0 \geq 1$ and $\delta > 0$ be given such that $- \lambda_n + q_c < -\delta < 0$ for all $n \geq N_0 +1$. Assume that 1) for any $1 \leq n \leq N_0$, there exists $1 \leq k = k(n) \leq m$ such that $b_{n,k} \neq 0$; 2) $c_n \neq 0$ for all $1 \leq n \leq N_0$. Let $K \in\R^{m \times N_0}$ and $L \in\R^{N_0}$ be such that $A_0 + B_0 K$ and $A_0 - L C_0$ are Hurwitz with eigenvalues that have a real part strictly less than $-\delta < 0$. For a given $N \geq N_0 +1$, assume that there exist a symmetric positive definite $P \in \R^{2N \times 2N}$, $\alpha,\beta,\gamma,\mu,\kappa > 0$, a diagonal positive definite $T \in \R^{m \times m}$, and $C \in \R^{m \times N_0}$ such that 
\begin{equation}\label{eq: thm1 constraints}
\Theta_1(\kappa) \preceq 0, \quad \Theta_2 \succeq 0, \quad\Theta_3(\kappa) \leq 0
\end{equation}
where
\begin{align*}
\Theta_1(\kappa) & = \begin{bmatrix} \Theta_{1,1,1}(\kappa) & P \mathcal{L} & - E^\top C^\top T + P \mathcal{L}_\phi \\ \mathcal{L}^\top P & -\beta & 0 \\ - T C E + \mathcal{L_\phi}^\top P & 0 & \alpha\gamma \sum_{k=1}^m \Vert \mathcal{R}_N b_k \Vert_{L^2}^2 I - 2T \end{bmatrix} \\
\Theta_2 & = \begin{bmatrix} P & E^\top (K-C)^\top \\ (K-C)E & \mu\,\mathrm{diag}(\ell)^2 \end{bmatrix} , \\
\Theta_3(\kappa) & = 2\gamma \left\{ -\lambda_{N+1} + q_c + \kappa + \frac{1}{\alpha} \right\} + \beta \Vert \mathcal{R}_N c \Vert_{L^2}^2
\end{align*}
with $\Theta_{1,1,1}(\kappa) = F^\top P + P F + 2 \kappa P + \alpha \gamma \sum_{k=1}^m \Vert \mathcal{R}_N b_k \Vert_{L^2}^2 \tilde{K}^\top\tilde{K}$. 
Consider the block representation $P = (P_{i,j})_{1 \leq i,j \leq 4}$ with dimensions that are compatible with (\ref{eq: state-vector closed-loop truncated model}) and define
\begin{align}
\mathcal{E}_1 & = \bigg\{
z \in L^2(0,1) : \nonumber \\
& \hspace{0.8cm} (\pi_{N}z)^\top 
\begin{bmatrix}
P_{2,2} & P_{2,4} \\ P_{4,2} & P_{4,4} 
\end{bmatrix}
(\pi_{N}z) + \gamma \Vert \mathcal{R}_N z \Vert_{L^2}^2
\leq \frac{1}{\mu}
\bigg\} . \label{eq: ellipsoid thm1}
\end{align}
Then, considering the closed-loop system composed of the plant (\ref{eq: distributed measurement - RD system})  with measured output (\ref{eq: bounded measurement operator}) and the control law (\ref{eq: control law}), there exists $M > 0$ such that for any initial condition $z_0 \in \mathcal{E}_1$ and with a zero initial condition of the observer (i.e., $\hat{z}_n(0)=0$ for all $1 \leq n \leq N$), the system trajectory satisfies 
\begin{equation}\label{eq: thm1 exponential stability}
\Vert z(t,\cdot) \Vert_{L^2}^2 + \sum_{n=1}^{N} \hat{z}_n(t)^2 \leq M e^{-2 \kappa t} \Vert z_0 \Vert_{L^2}^2
\end{equation}
for all $t \geq 0$. Moreover, for any fixed $\kappa \in (0,\delta]$, the above constraints are always feasible for $N$ large enough.
\end{theorem}

\begin{proof}
In order to use a Lyapunov-based argument, we start by considering classical solutions. The result for mild solutions will then be obtained by a density argument. Let $z_0 \in D(\mathcal{A})$ be an initial condition and let $(z,\hat{Z}^{N_0},\hat{Z}^{N-N_0}) \in \mathcal{C}^0(\R_{\geq 0} ; D(\mathcal{A}) \times \R^N ) \cap \mathcal{C}^1(\R_{\geq 0} ; L^2(0,1) \times \R^N )$ be the associated classical solution of the closed-loop system with zero initial conditions for the observer, whose existence is obtained from~\cite[Thm.~6.1.2, Thm.~6.1.6, Cor.~4.2.11]{pazy2012semigroups}. Define the Lyapunov function candidate $V(X,z) = X^\top P X + \gamma \sum_{n \geq N+1} \left< z , \phi_n \right>^2$ for $X\in\R^{2N}$ and $z \in L^2(0,1)$ (see e.g. \cite{coron2004global}). The computation of the time derivative of $V$ along the system trajectories (\ref{eq: distributed measurement - projection}) and (\ref{eq: dynamics truncated model}) gives
\begin{align*}
& \dot{V} + 2 \kappa V = X^\top \left( F^\top P + P F + 2 \kappa P \right) X + 2 X^\top P \mathcal{L} \zeta \\
& + 2 X^\top P \mathcal{L}_\phi \phi_\ell(K\hat{Z}^{N_0}) + 2 \gamma \sum_{n \geq N+1} (-\lambda_n + q_c + \kappa) z_n^2 \\
& + 2 \gamma \sum_{n \geq N+1} z_n \mathcal{L}_{b,n} \tilde{K}X + 2 \gamma \sum_{n \geq N+1} z_n \mathcal{L}_{b,n} \phi_\ell(K\hat{Z}^{N_0}) 
\end{align*}
where $\mathcal{L}_{b,n} = \begin{bmatrix} b_{n,1} & \ldots & b_{n,m} \end{bmatrix}$. Using Young's inequality, we have for any $\alpha > 0$ and any $w \in \R^m$ that
\begin{align*}
2 \sum_{n \geq N+1} z_n \mathcal{L}_{b,n} w
& \leq \dfrac{1}{\alpha} \sum_{n \geq N+1} z_n^2 + \alpha \sum_{k=1}^m \Vert \mathcal{R}_N b_k \Vert_{L^2}^2 \Vert w \Vert^2 .
\end{align*}
Hence, defining $\tilde{X} = \mathrm{col} (X , \zeta , \phi_\ell(K\hat{Z}^{N_0}) )$, we infer that
\begin{align*}
& \dot{V} + 2 \kappa V \leq \tilde{X}^\top \begin{bmatrix} \Theta_{1,1,1}(\kappa) & P\mathcal{L} & P\mathcal{L}_\phi \\ \mathcal{L}^\top P & 0 & 0 \\ \mathcal{L}_\phi^\top P & 0 & \alpha\gamma\sum_{k=1}^m \Vert \mathcal{R}_N b_k \Vert_{L^2}^2 I \end{bmatrix} \tilde{X} \\
& \phantom{\dot{V} + 2 \kappa V \leq}\, + 2 \gamma \sum_{n \geq N+1} \left( -\lambda_n + q_c + \kappa + \dfrac{1}{\alpha} \right) z_n^2 .
\end{align*}
Recalling that $\zeta = \sum_{n \geq N+1} c_n z_n$, we obtain that $\zeta^2 \leq \Vert \mathcal{R}_N c \Vert_{L^2}^2 \sum_{n \geq N+1} z_n^2$. Hence, for any $\beta > 0$, $\beta \Vert \mathcal{R}_N c \Vert_{L^2}^2 \sum_{n \geq N+1} z_n^2 - \beta \zeta^2 \geq 0$. Assuming that $\hat{Z}^{N_0} \in \R^{N_0}$ satisfies $\vert (K-C)\hat{Z}^{N_0} \vert \leq \ell$, we also have from Lemma~\ref{eq: generalized sector condition} that 
\begin{equation*}
\phi_\ell(K\hat{Z}^{N_0})^\top T ( \phi_\ell(K\hat{Z}^{N_0}) + C\hat{Z}^{N_0} ) \leq 0 .
\end{equation*}
Combining the three latter inequalities, defining $\Gamma_n = 2 \gamma \left( -\lambda_{n} + q_c + \kappa + \frac{1}{\alpha} \right) + \beta \Vert \mathcal{R}_N c \Vert_{L^2}^2$, and recalling that $\hat{Z}^{N_0} = E X$, we obtain that 
\begin{align*}
& \dot{V} + 2 \kappa V
\leq \tilde{X}^\top \Theta_1(\kappa) \tilde{X} 
+ \sum_{n \geq N+1} \Gamma_{n} z_n^2 
\end{align*}
for all $X \in \R^{2N}$ satisfying $\vert (K-C)EX \vert \leq \ell$. Since $\Gamma_n \leq \Theta_3(\kappa) \leq 0$ for all $n \geq N+1$ and $\Theta_1(\kappa) \preceq 0$, we infer that $\dot{V} + 2 \kappa V \leq 0$ holds as soon as $X \in \R^{2N}$ is such that $\vert (K-C)EX \vert \leq \ell$. 

Consider $X\in\R^{2N}$ and $z \in L^2(0,1)$ such that $V(X,z) \leq 1/\mu$. By Schur complement, $\Theta_2 \succeq 0$ implies that $P \succeq \frac{1}{\mu} E^\top (K-C)^\top \mathrm{diag}(\ell)^{-2} (K-C)E$. Therefore we have $\Vert \mathrm{diag}(\ell)^{-1} (K-C)EX \Vert \leq 1$. This in particular implies that $\vert (K-C)EX \vert \leq \ell$ hence, in the context of the previous paragraph, $\dot{V} + 2 \kappa V \leq 0$. 

Assume now that the initial condition is selected such that $z_0 \in D(\mathcal{A}) \cap \mathcal{E}_1$. Since the initial condition of the observer is zero, this implies that $X(0) = \mathrm{col}(0,\pi_{N_0}z_0,0,\pi_{N_0,N}z_0)$ and $V(X(0),z_0) \leq 1/\mu$. Assuming that $z_0 \neq 0$ (otherwise the system trajectory is identically zero), we obtain that $\dot{V}(X(0),z_0) \leq - 2\kappa V(X(0),z_0) < 0$. A simple contradiction argument shows that $V(X(t),z(t,\cdot)) \leq 1/\mu$ hence $\dot{V}(X(t),z(t,\cdot)) + 2\kappa V(X(t),z(t,\cdot)) \leq 0$ for all $t \geq 0$. We deduce that $V(X(t),z(t,\cdot)) \leq e^{-2\kappa t} V(X(0),z_0)$ for all $t \geq 0$ and all $z_0 \in D(\mathcal{A}) \cap \mathcal{E}_1$. The claimed estimate now easily follows for classical solutions from the definition of $V$. The result for mild solutions associated with any $z_0 \in \mathcal{E}_1$ follows from a classical density argument~\cite[Thm.~6.1.2]{pazy2012semigroups}.

It remains to show that for any given $\kappa \in (0,\delta]$, the constraints are always feasible for $N$ sufficiently large. We first set $C=0$. Now, we note that the gains $K$ and $L$ are independent of $N$ while $\Vert C_1 \Vert = O(1)$ and $\Vert B_1 \Vert = O(1)$ because $c,b_{k} \in L^2(0,1)$. The matrices $A_0+B_0 K + \kappa I$ and $A_0 - L C_0 + \kappa I$ are Hurwitz. Finally $\Vert \exp((A_1+\kappa I) t) \Vert \leq e^{-\kappa_0 t}$ for all $t \geq 0$ and all $N \geq N_0 +1$ with $\kappa_0 = \lambda_{N_0+1}-q_c-\kappa \geq \lambda_{N_0+1}-q_c-\delta > 0$. Hence, the application of the Lemma reported in Appendix to the matrix $F + \kappa I$ ensures the existence of $Q \succ 0$ such that $F^\top Q + Q F + 2 \kappa Q = - I$ with $\Vert Q \Vert = O(1)$ as $N \rightarrow + \infty$. We now define $P = \eta Q$ for some $\eta > 0$ to be defined and the matrix
\begin{align*}
\Theta_{1p} 
& = \begin{bmatrix} \Theta_{1,1,1}(\kappa) & P \mathcal{L} \\ \mathcal{L}^\top P & - \beta \end{bmatrix} \\
& = \begin{bmatrix} -\eta I + \alpha \gamma \sum_{k=1}^m \Vert \mathcal{R}_N b_k \Vert_{L^2}^2 \tilde{K}^\top\tilde{K} & \eta Q \mathcal{L} \\ \eta \mathcal{L}^\top Q & - \beta \end{bmatrix} .
\end{align*}
Invoking Schur complement, $\Theta_{1p} \prec 0$ if and only if $-\eta I + \alpha \gamma \sum_{k=1}^m \Vert \mathcal{R}_N b_k \Vert_{L^2}^2 \tilde{K}^\top\tilde{K} + \frac{\eta^2}{\beta} Q\mathcal{L}\mathcal{L}^\top Q \prec 0$. Fixing an arbitrary value for $\alpha > 0$ while setting $\beta = \eta N$ and $\gamma = \eta/\sqrt{N}$, we obtain that $\Theta_{1p} \prec 0$ and $\Theta_{3}(\kappa) \leq 0$ if and only if
\begin{align*}
& - I + \frac{\alpha}{\sqrt{N}} \sum_{k=1}^m \Vert \mathcal{R}_N b_k \Vert_{L^2}^2 \tilde{K}^\top\tilde{K} + \frac{1}{N} Q\mathcal{L}\mathcal{L}^\top Q \prec 0 , \\
& - \lambda_{N+1} + q_c + \kappa + \frac{1}{\alpha} + N \sqrt{N} \frac{\Vert \mathcal{R}_N c \Vert_{L^2}^2}{2} \leq 0 .
\end{align*}  
Based on (\ref{eq: estimation lambda_n}) and noting that $\Vert \mathcal{L}\Vert = \sqrt{2} \Vert L \Vert$ and $\Vert \tilde{K} \Vert = \Vert K \Vert$ while recalling that $\Vert Q \Vert = O(1)$, we obtain the existence of a sufficiently large $N$, selected independently of $\eta,\mu > 0$ and of the matrix $T$, such that $\Theta_{1p} \prec 0$ and $\Theta_3(\kappa) \leq 0$. This fixes the dimension $N \geq N_0 + 1$ as well as the matrix $Q$. We now consider the constraint $\Theta_2 \succeq 0$, which is equivalent to $\eta Q - \frac{1}{\mu} E^\top K^\top \mathrm{diag}(\ell)^{-2} K E \succeq 0$ by Schur complement. We fix an arbitrary value of $\mu > 0$. Since $Q \succ 0$,we fix $\eta > 0$ large enough such that we indeed have $\Theta_2 \succeq 0$. This definitely fixes the decision variables $P = \eta Q \succ 0$, $\beta = \eta N > 0$, and $\gamma = \eta/\sqrt{N} > 0$. To conclude, it remains to tune the diagonal positive definite matrix $T \in \R^{m \times m}$ in order to ensure that $\Theta_1(\kappa) \preceq 0$. Imposing $T = \tau I$ for $\tau > \frac{\alpha\gamma}{2}\sum_{k=1}^m \Vert \mathcal{R}_N b_k \Vert_{L^2}^2$, Schur complement shows that $\Theta_1(\kappa) \preceq 0$ if and only if
\begin{equation*}
\Theta_{1p} + \frac{1}{2\tau - \alpha\gamma \sum_{k=1}^m \Vert \mathcal{R}_N b_k \Vert_{L^2}^2} 
\begin{bmatrix} P \mathcal{L}_\phi \\ 0 \end{bmatrix}
\begin{bmatrix} P \mathcal{L}_\phi \\ 0 \end{bmatrix}^\top
\preceq 0 .
\end{equation*}
Since $\Theta_{1p} \prec 0$ is independent of $\tau$, we obtain that the latter inequality is satisfied for $\tau >0$ large enough. Hence, we have achieved an adjustment of the different degrees of freedom such that $\Theta_1(\kappa) \preceq 0$, $\Theta_2 \succeq 0$, $\Theta_3(\kappa) \leq 0$. This completes the proof.
\end{proof}

\begin{remark}
In~\cite{mironchenko2020local} the control law takes the form of the state-feedback $u = K \pi_{N_0} z(t,\cdot)$. In the presence of saturations, the region of attraction derived therein takes the form of $\mathcal{E} = \{ z \in L^2(0,1) \,:\, (\pi_{N_0}z)^\top P (\pi_{N_0}z) \leq 1 \}$ where $P \succ 0$ satisfies suitable LMI conditions. Hence the presence of a saturation mechanism only imposes constraints on the $N_0$ first modes of the reaction-diffusion plant. In sharp contrast, the presence of a saturation mechanism for the output feedback controller (\ref{eq: control law}) imposes constraints on all the modes of the PDE as seen from (\ref{eq: ellipsoid thm1}).
\end{remark}

\begin{remark}\label{rmk: thm1 - rmk0}
For any given $\kappa > 0$, the exponential stability estimate (\ref{eq: thm1 exponential stability}) with decay rate $\kappa$ can always be achieved for initial conditions in a neighborhood of the origin of the form (\ref{eq: ellipsoid thm1}) by setting $\delta = \kappa$ and by selecting the integer $N$ large enough.
\end{remark}

\begin{remark}\label{rmk: thm1 - rmk1}
Considering a given $N \geq N_0 + 1$ and $0 < \kappa \leq \overline{\kappa}$, it is easy to see that $\Theta_1(\kappa) \preceq \Theta_1(\overline{\kappa})$ and $\Theta_3(\kappa) \leq \Theta_3(\overline{\kappa})$. Therefore, the feasibility of the constraints of Theorem~\ref{thm: main result 1} with decay rate $\overline{\kappa} > 0$ implies the feasibility of the constraints with the same value of the decision variables $P,\alpha,\beta,\gamma,\mu,T,C$ for all $\kappa \in (0,\overline{\kappa}]$.
\end{remark}

\begin{remark}\label{rmk: thm1 - rmk2}
In the context of Theorem~\ref{thm: main result 1}, let $N \geq N_0 + 1$, a symmetric positive definite $P \in \R^{2N \times 2N}$, $\alpha,\beta,\gamma,\mu > 0$, a diagonal positive definite $T \in \R^{m \times m}$, and $C \in \R^{m \times N_0}$ such that\footnote{Existence is guaranteed by the last part of the proof of Theorem~\ref{thm: main result 1} which remains valid in the case $\kappa = 0$.} 
\begin{equation}\label{eq: thm1 constraints - strict version}
\Theta_1(0) \prec 0, \quad \Theta_2 \succeq 0, \quad \Theta_3(0) < 0 .
\end{equation}
Then a continuity argument at $\kappa = 0$ shows the existence of $\kappa > 0$ such that $\Theta_1(\kappa) \preceq 0$, $\Theta_2 \succeq 0$, $\Theta_3(\kappa) \leq 0$, allowing the application of the conclusions of Theorem~\ref{thm: main result 1}.
\end{remark}

\begin{remark}\label{rmk: thm1 - rmk3}
Even if stated for a zero initial condition $\hat{z}_0 = 0$ of the observer, the conclusions of Theorem~\ref{thm: main result 1} can actually be extended in a straightforward manner to non zero initial conditions $\hat{z}_0 \in\R^N$. In that case, following the proof of Theorem~\ref{thm: main result 1}, we obtain the exponential decay to the origin of the system trajectories as soon as the initial conditions satisfy $V(X(0),z_0) \leq 1/\mu$ with $X(0) = \mathrm{col}(\hat{z}_{0,1},\pi_{N_0}z_0-\hat{z}_{0,1},\hat{z}_{0,2},\pi_{N_0,N}z_0-\hat{z}_{0,2})$ where $\hat{z}_0 = \mathrm{col}(\hat{z}_{0,1},\hat{z}_{0,2})$, $\hat{z}_{0,1} \in \R^{N_0}$, and $\hat{z}_{0,2} \in \R^{N-N_0}$.
\end{remark}

\begin{remark}
The quantity $\Vert \mathcal{R}_N z \Vert_{L^2}^2$ appearing in the definition of the ellipsoid $\mathcal{E}_1$ defined by (\ref{eq: ellipsoid thm1}) can be easily computed in practice by noting that $\Vert \mathcal{R}_N z \Vert_{L^2}^2 = \Vert z \Vert_{L^2}^2 - \Vert \pi_N z \Vert^2$. 
\end{remark}

\subsubsection{Stabilization in $H^1$ norm}

The following result deals with the exponential stability of the system trajectories evaluated in $H^1$ norm.

\begin{theorem}\label{thm: main result 2}
In the context of Theorem~\ref{thm: main result 1}, we further constrain $q \in \mathcal{C}^0([0,1])$ and $q_c \in\R$ such that (\ref{eq: decomposition of the reaction term}) holds with the additional constraint $q > 0$. For a given $N \geq N_0 +1$, assume that there exist a symmetric positive definite $P \in \R^{2N \times 2N}$, $\alpha > 1$, $\beta,\gamma,\mu,\kappa > 0$, a diagonal positive definite $T \in \R^{m \times m}$, and $C \in \R^{m \times N_0}$ such that 
\begin{equation}\label{eq: thm2 constraints}
\Theta_1(\kappa) \preceq 0, \quad \Theta_2 \succeq 0, \quad\Theta_3(\kappa) \leq 0
\end{equation}
where $\Theta_1(\kappa)$ and $\Theta_2$ are defined as in Theorem~\ref{thm: main result 1} while 
\begin{align*}
\Theta_3(\kappa) & = 2\gamma \left\{ - \left( 1 - \frac{1}{\alpha} \right) \lambda_{N+1} + q_c + \kappa \right\} + \frac{\beta \Vert \mathcal{R}_N c \Vert_{L^2}^2}{\lambda_{N+1}} .
\end{align*}
Consider the block representation $P = (P_{i,j})_{1 \leq i,j \leq 4}$ with dimensions that are compatible with (\ref{eq: state-vector closed-loop truncated model}) and define
\begin{align}
\mathcal{E}_2 & = \bigg\{
z \in D(\mathcal{A}) : (\pi_{N}z)^\top 
\begin{bmatrix}
P_{2,2} & P_{2,4} \\ P_{4,2} & P_{4,4} 
\end{bmatrix}
(\pi_{N}z) \nonumber \\
& \hspace{2.4cm} + \gamma \Vert \mathcal{R}_N \mathcal{A}^{1/2}z \Vert_{L^2}^2
\leq \frac{1}{\mu}
\bigg\} . \label{eq: ellipsoid thm2}
\end{align}
Then, considering the closed-loop system composed of the plant (\ref{eq: distributed measurement - RD system}) with measured output (\ref{eq: bounded measurement operator}) and the control law (\ref{eq: control law}), there exists $M > 0$ such that for any initial condition $z_0 \in \mathcal{E}_2$ and with a zero initial condition of the observer (i.e., $\hat{z}_n(0)=0$ for all $1 \leq n \leq N$), the system trajectory satisfies $\Vert z(t,\cdot) \Vert_{H^1}^2 + \sum_{n=1}^{N} \hat{z}_n(t)^2 \leq M e^{-2 \kappa t} \Vert z_0 \Vert_{H^1}^2$ for all $t \geq 0$. Moreover, for any fixed $\kappa \in (0,\delta]$, the above constraints are always feasible for $N$ large enough.
\end{theorem}

\begin{proof}
Consider the Lyapunov functional candidate $V(X,z) = X^\top P X + \gamma \sum_{n \geq N+1} \lambda_n \left< z , \phi_n \right>^2$ defined for $X\in\R^{2N}$ and $z \in D(\mathcal{A})$ (see e.g. \cite{coron2004global}). The computation of the time derivative of $V$ along the system trajectories (\ref{eq: distributed measurement - projection}) and (\ref{eq: dynamics truncated model}) gives
\begin{align*}
& \dot{V} + 2 \kappa V = X^\top \left( F^\top P + P F + 2 \kappa P \right) X + 2 X^\top P \mathcal{L} \zeta \\
& + 2 X^\top P \mathcal{L}_\phi \phi_\ell(K\hat{Z}^{N_0}) + 2 \gamma \sum_{n \geq N+1} \lambda_n (-\lambda_n + q_c + \kappa) z_n^2 \\
& + 2 \gamma \sum_{n \geq N+1} \lambda_n z_n \mathcal{L}_{b,n} \tilde{K}X + 2 \gamma \sum_{n \geq N+1} \lambda_n z_n \mathcal{L}_{b,n} \phi_\ell(K\hat{Z}^{N_0}) 
\end{align*}
where $\mathcal{L}_{b,n} = \begin{bmatrix} b_{n,1} & \ldots & b_{n,m} \end{bmatrix}$. Using Young's inequality, we have for any $\alpha > 0$ and any $w \in \R^m$ that
\begin{align*}
2 \sum_{n \geq N+1} \lambda_n z_n \mathcal{L}_{b,n} w
& \leq \dfrac{1}{\alpha} \sum_{n \geq N+1} \lambda_n^2 z_n^2 + \alpha \sum_{k=1}^m \Vert \mathcal{R}_N b_k \Vert_{L^2}^2 \Vert w \Vert^2 .
\end{align*}
Introducing the vector $\tilde{X} = \mathrm{col} (X , \zeta , \phi_\ell(K\hat{Z}^{N_0}) )$ and the quantity $\Gamma_n = 2\gamma \left\{ - \left( 1 - \frac{1}{\alpha} \right) \lambda_{n} + q_c + \kappa \right\} + \frac{\beta \Vert \mathcal{R}_N c \Vert_{L^2}^2}{\lambda_{n}}$ for $n \geq N+1$, we infer similarly to the proof of Theorem~\ref{thm: main result 1} that 
\begin{align*}
& \dot{V} + 2 \kappa V
\leq \tilde{X}^\top \Theta_1(\kappa) \tilde{X} 
+ \sum_{n \geq N+1} \lambda_n \Gamma_{n} z_n^2 
\end{align*}
for all $X \in \R^{2N}$ satisfying $\vert (K-C)EX \vert \leq \ell$. Since $\alpha > 1$, we infer that $\Gamma_n \leq \Theta_3(\kappa) \leq 0$ for all $n \geq N+1$. Combining this with $\Theta_1(\kappa) \preceq 0$, we obtain that $\dot{V} + 2 \kappa V \leq 0$ for all $X \in \R^{2N}$ such that $\vert (K-C)EX \vert \leq \ell$. Following now similar arguments that the ones employed in the proof of Theorem~\ref{thm: main result 1}, we obtain that  $V(X(t),z(t,\cdot)) \leq e^{-2\kappa t} V(X(0),z_0)$ for all $t \geq 0$ and all $z_0 \in \mathcal{E}_2$ while considering zero initial conditions for the observer dynamics. The claimed estimate now follows from the definition of $V$ and (\ref{eq: inner product Af and f}). The feasibility of the constraints for $N$ large enough follows the same arguments that the ones reported in the proof of Theorem~\ref{thm: main result 1}. 
\end{proof}

\begin{remark}
Note that similar remarks to the ones stated in Remarks~\ref{rmk: thm1 - rmk0}, \ref{rmk: thm1 - rmk1}, \ref{rmk: thm1 - rmk2} and~\ref{rmk: thm1 - rmk3} also apply to Theorem~\ref{thm: main result 2}. Moreover, the quantity $\Vert \mathcal{R}_N \mathcal{A}^{1/2} z \Vert_{L^2}^2$ appearing in the definition (\ref{eq: ellipsoid thm2}) of the ellipsoid $\mathcal{E}_2$ can be easily computed in practice by noting that $\Vert \mathcal{R}_N \mathcal{A}^{1/2} z \Vert_{L^2}^2 = \left< \mathcal{A}z , z \right> - \sum_{n=1}^{N} \lambda_n \left< z , \phi_n \right>^2$ while an integration by parts shows that $\left< \mathcal{A}z , z \right> = p(0)z(0)z'(0) - p(1)z(1)z'(1) + \int_0^1 p (z')^2 + q z^2 \,\mathrm{d}x$. 
\end{remark}

\subsection{Numerical considerations}

\subsubsection{Derivation of LMI conditions}\label{subsubsec: Derivation of LMI conditions}

For a given decay rate $\kappa \in (0,\delta]$ and a given number of observed modes $N \geq N_0 +1$, the constraints of either Theorem~\ref{thm: main result 1} or Theorem~\ref{thm: main result 2} are nonlinear functions of the decision variables $P,\alpha,\beta,\gamma,\mu,T,C$. We propose here to reformulate these constraints in order to obtain LMIs, allowing the use of efficient numerical tools for a fixed value of $N$. 

First, one can arbitrarily fix the value of $\alpha > 0$ in the case of Theorem~\ref{thm: main result 1} and $\alpha > 1$ in the case of Theorem~\ref{thm: main result 2}. Even if this removes one degree of freedom, the constraints of either Theorem~\ref{thm: main result 1} or Theorem~\ref{thm: main result 2} still remain feasible for a sufficiently large $N$ as shown in the associated proofs. By doing so, the constraints take the form of BMIs of the decision variables $P,\beta,\gamma,\mu,T,C$.

Second, in order to obtain LMIs, we further constrain the decision variables by imposing $T = \tau T_0$ for some given diagonal positive definite matrix $T_0 \in \R^{m \times m}$ while $\tau > 0$ stands for the new decision variable. Again, by a slight adaptation of the proof of Theorem~\ref{thm: main result 1}, the resulting constraints are still feasible for $N$ large enough. Introducing the change of variable $\tilde{C} = \tau C$, we infer that
\begin{align*}
& \Theta_1(\kappa) = \\
& \begin{bmatrix} \Theta_{1,1,1}(\kappa) & P \mathcal{L} & - E^\top \tilde{C}^\top T_0 + P \mathcal{L}_\phi \\ \mathcal{L}^\top P & -\beta & 0 \\ - T_0 \tilde{C} E + \mathcal{L}_\phi^\top P & 0 & \alpha\gamma \sum_{k=1}^m \Vert \mathcal{R}_N b_k \Vert_{L^2}^2 I - 2 \tau T_0 \end{bmatrix} .
\end{align*}
Moreover, defining $\tilde{\mu} = \tau^2 \mu$ and
\begin{align*}
\tilde{\Theta}_2 =
\begin{bmatrix}
I & 0 \\ 0 & \tau I
\end{bmatrix}^\top 
\Theta_2
\begin{bmatrix}
I & 0 \\ 0 & \tau I
\end{bmatrix}
= \begin{bmatrix} P & E^\top (\tau K - \tilde{C})^\top \\ (\tau K - \tilde{C})E & \tilde{\mu} \,\mathrm{diag}(\ell)^2 \end{bmatrix} 
\end{align*}
we have $\Theta_2 \succeq 0$ if and only if $\tilde{\Theta}_2 \succeq 0$. Hence, with fixed $\alpha,T_0$, the constraints reduce to the LMIs $\Theta_1(\kappa) \preceq 0$, $\tilde{\Theta}_2 \succeq 0$, $\Theta_3(\kappa) \leq 0$ with decision variables $P,\beta,\gamma,\tilde{\mu},\tau,\tilde{C}$. If feasibles, one can recover the original variables by setting $C = \tilde{C}/\tau$, $T = \tau T_0$, and $\mu = \tilde{\mu} / \tau^2$.

\begin{remark}
In the case of a single input channel ($m=1$), the diagonal positive definite matrix $T$ reduces to a positive constant. Hence the application of the above procedure only imposes to fix $\alpha > 0$ in the case of Theorem~\ref{thm: main result 1} and $\alpha > 1$ in the case of Theorem~\ref{thm: main result 2} in order to obtain the constraints in the form of LMIs.
\end{remark}

\subsubsection{Estimation of the domain of attraction}\label{subsubsec: shaping domain of attraction}

In the context of Theorem~\ref{thm: main result 1}, let an integer $N \geq N_0 + 1$ and a $\kappa \in (0,\delta]$ such that the associated constraints (\ref{eq: thm1 constraints}) are feasible. Consider a given symmetric positive definite matrix $R \in \R^{(N+1) \times (N+1)}$. Let $r > 0$ be such that
\begin{equation}\label{eq: constraint to shape the doamin of attraction}
\mathcal{P} \triangleq
\begin{bmatrix}
P_{2,2} & P_{2,4} & 0 \\ P_{4,2} & P_{4,4} & 0 \\ 0 & 0 & \gamma 
\end{bmatrix}
\preceq \frac{r}{\mu} R .
\end{equation}
under the constraints (\ref{eq: thm1 constraints}). Note that $\mathcal{P} \preceq \frac{r}{\mu}R$ can always be achieved by selecting $r > 0$ large enough. In this case we have 
\begin{align*}
& \begin{bmatrix} \pi_N z \\ \Vert \mathcal{R}_N z \Vert_{L^2} \end{bmatrix}^\top r R \begin{bmatrix} \pi_N z \\ \Vert \mathcal{R}_N z \Vert_{L^2} \end{bmatrix} 
\leq 1 \\
& \qquad\Rightarrow
\begin{bmatrix} \pi_N z \\ \Vert \mathcal{R}_N z \Vert_{L^2} \end{bmatrix}^\top
\mathcal{P}
\begin{bmatrix} \pi_N z \\ \Vert \mathcal{R}_N z \Vert_{L^2} \end{bmatrix}
\leq \frac{1}{\mu}
\end{align*}
i.e., 
\begin{equation}\label{eq: minimization radius r - L2 norm}
\left\{ z \in L^2(0,1) : 
\begin{bmatrix} \pi_N z \\ \Vert \mathcal{R}_N z \Vert_{L^2} \end{bmatrix}^\top R \begin{bmatrix} \pi_N z \\ \Vert \mathcal{R}_N z \Vert_{L^2} \end{bmatrix}  \leq \frac{1}{r}
\right\} 
\subset \mathcal{E}_1
\end{equation}
where $\mathcal{E}_1$ is given by (\ref{eq: ellipsoid thm1}). In this context, for given $N \geq N_0 +1$ and $\kappa \in (0,\delta]$ such that the constraints (\ref{eq: thm1 constraints}) have been found feasible, we are interested in minimizing $r > 0$ under the constraints (\ref{eq: thm1 constraints}) and (\ref{eq: constraint to shape the doamin of attraction}) with decision variables $P,\alpha,\beta,\gamma,\mu,T,C,r$. 

Note that the obtained minimization problem is nonlinear. In practice, it is easier to solve iteratively the following sub-optimal problem. First, we fix the values of $\alpha,\mu > 0$ to values associated with a feasible solution of the constraints (\ref{eq: thm1 constraints}). By doing so, the only nonlinearity is the product term $TC$ in $\Theta_1(\kappa) \preceq 0$, making the problem bilinear. Second, one can successively fix either the value of $T$ or $C$ to its previously computed value in order to iteratively minimize the value of $r > 0$ under LMI constraints. This approach, although sub-optimal, has the merit to be numerically efficient.

The above procedure also applies to the setting of Theorem~\ref{thm: main result 2}. In that case, (\ref{eq: constraint to shape the doamin of attraction}) under the constraints (\ref{eq: thm2 constraints}) implies that 
\begin{equation}\label{eq: minimization radius r - H1 norm}
\left\{ z \in D(\mathcal{A}) : 
\begin{bmatrix} \pi_N z \\ \Vert \mathcal{R}_N \mathcal{A}^{1/2} z \Vert_{L^2} \end{bmatrix}^\top R \begin{bmatrix} \pi_N z \\ \Vert \mathcal{R}_N \mathcal{A}^{1/2} z \Vert_{L^2} \end{bmatrix}  \leq \frac{1}{r}
\right\} 
\subset \mathcal{E}_2
\end{equation}
where $\mathcal{E}_2$ is given by (\ref{eq: ellipsoid thm2}).

\section{Numerical example}\label{sec: numerical example}

We illustrate the theoretical results of Theorem~\ref{thm: main result 1} and Theorem~\ref{thm: main result 2} in the case of Dirichlet boundary conditions ($\theta_1 = \theta_2 = 0$) with $p = 1$ and $\tilde{q} = - 10$, yielding an unstable open-loop reaction-diffusion PDE. We consider the case of $m = 2$ control inputs characterized by $b_1(x) = \cos(x) 1_{[1/10,3/10]}(x)$ and $b_2(x) = -(1/2+x) 1_{[7/10,9/10]}(x)$. The associated saturation levels are set as $l_1 = 1$ and $l_2 = 2$. The measured output is characterized by $c(x) = 1_{[9/20,11/20]}(x)$.

We select $q = 1 > 0$ and $q_c = 11$ which are such that (\ref{eq: decomposition of the reaction term}) holds. We set $\delta = 1$ giving $N_0 = 1$. The feedback and observation gains are set as $K = \begin{bmatrix} 2.59 & 3.41 \end{bmatrix}^\top$ and $L = 15.13$. The application of the procedure reported in Subsection~\ref{subsubsec: Derivation of LMI conditions} shows the feasibility of (\ref{eq: thm1 constraints - strict version}) for $N = 4$ in the cases of both Theorem~\ref{thm: main result 1} and Theorem~\ref{thm: main result 2}. We now set $R = \mathrm{diag}(I_4,0.005)$ and we aim at minimizing the value of $r > 0$ such that (\ref{eq: minimization radius r - L2 norm}) and (\ref{eq: minimization radius r - H1 norm}) hold simultaneously. Applying the procedure reported in Subsection~\ref{subsubsec: shaping domain of attraction}, the value of $r > 0$ decreases from $7.27$ to $0.41$.

We consider the initial condition $z_0(x) = 8.5 x(1-x)$ which is such that $z_0 \in \mathcal{E}_1 \cap \mathcal{E}_2$, ensuring based on Theorems~\ref{thm: main result 1} and~\ref{thm: main result 2} the exponential decay of the system trajectories to zero in both $L^2$ and $H^1$ norms. The behavior of the closed-loop system (obtained based on the 50 dominant modes of the PDE) associated with this initial condition is depected in Fig~\ref{fig: sim1}. It can been seen that both the state of the PDE and the observation error converge to zero in spite of significant saturations in both scalar control input channels $u_1(t)$ and $u_2(t)$. This numerical result shows that Theorems~\ref{thm: main result 1} and~\ref{thm: main result 2} can be used to assess the exponential decay of certain system trajectories of the closed-loop plant even when the saturation mechanism is actively solicited during the transient.

\begin{figure}
\centering
\subfigure[State $z(t,x)$]{
\includegraphics[width=3.25in]{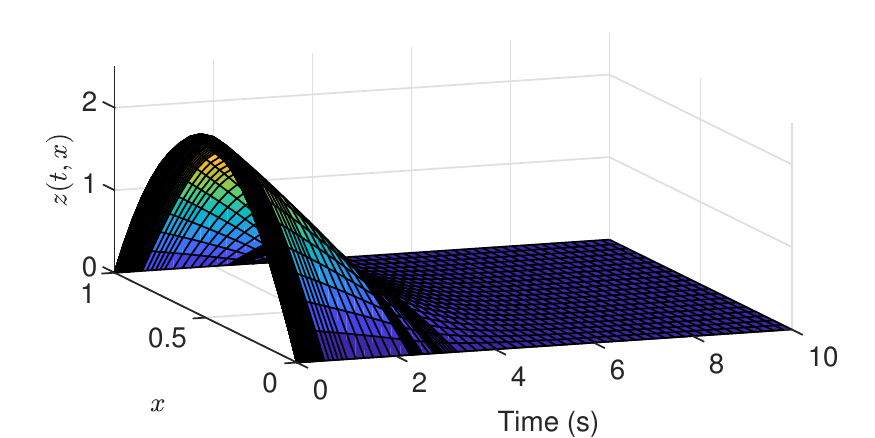}
}
\subfigure[Observation error $z(t,x)-\sum_{n=1}^{N} \hat{z}_n(t) \phi_n(x)$]{
\includegraphics[width=3.25in]{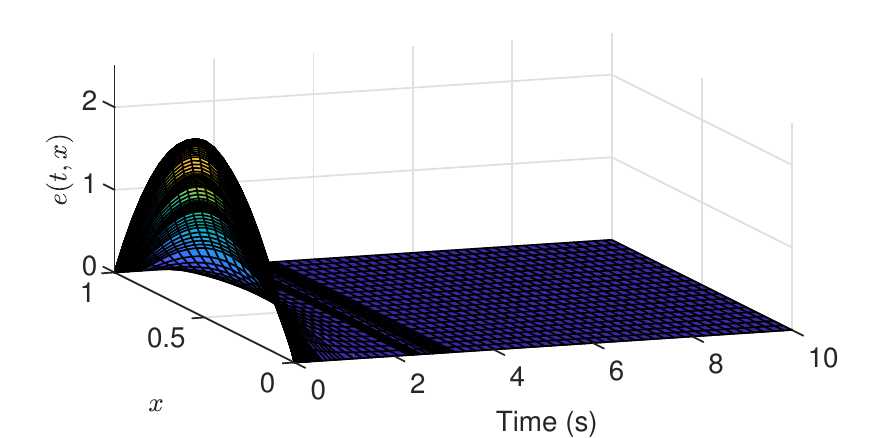}
}
\subfigure[Control input $u(t)$]{
\includegraphics[width=3.25in]{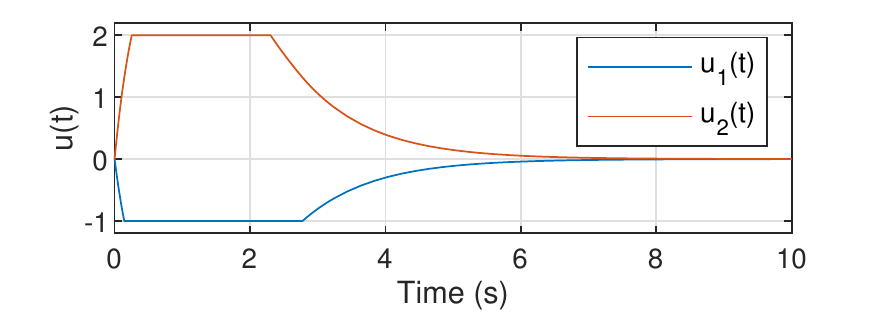}
}
\caption{Closed-loop behavior for the initial condition $z_0(x) = 8.5 x(1-x)$}
\label{fig: sim1}
\end{figure}

\section{Extension to boundary measurements}\label{sec: extension}

Assume that $p \in \mathcal{C}^2([0,1])$ and consider instead of the bounded measurement operator (\ref{eq: bounded measurement operator}) that the observation is made by either the left Dirichlet boundary measurement
\begin{equation}\label{eq: Dirichlet boundary measurement}
y(t) = z(t,0) = \sum_{n \geq 1} c_n z_n(t)
\end{equation} 
with $c_n = \phi_n(0)$ in the case $\theta_1 \in (0,\pi/2]$ or the left Neumann boundary measurement
\begin{equation}\label{eq: Neumann boundary measurement}
y(t) = z_x(t,0) = \sum_{n \geq 1} c_n z_n(t) 
\end{equation} 
with $c_n = \phi'_n(0)$ in the case $\theta_1 \in [0,\pi/2)$, where the above series expansion hold for classical solutions. We reuse the control architecture (\ref{eq: control law}) and the control design procedure presented in Subsection~\ref{subsec: control architecture} with the updated version of $c_n \in\R$. Note that the existence of classical solutions associated with initial conditions $z_0 \in D(\mathcal{A})$ for the subsequent closed-loop system is the consequence of~\cite[Sec.~6.3]{pazy2012semigroups}. After replacement of the definition of $E^{N-N_0}$ and $C_1$ by 
\begin{align*}
E^{N-N_0} & = \begin{bmatrix} \sqrt{\lambda_{N_0+1}} e_{N_0 + 1} & \ldots & \sqrt{\lambda_{N}} e_{N} \end{bmatrix}^\top , \\
C_1 & = \begin{bmatrix} \frac{c_{N_0 + 1}}{\sqrt{\lambda_{N_0 +1}}} & \ldots & \frac{c_{N + 1}}{\sqrt{\lambda_{N +1}}} \end{bmatrix}
\end{align*}
in the case of the Dirichlet boundary measurement (\ref{eq: Dirichlet boundary measurement}) and by 
\begin{align*}
E^{N-N_0} & = \begin{bmatrix} \lambda_{N_0+1} e_{N_0 + 1} & \ldots & \lambda_{N} e_{N} \end{bmatrix}^\top , \\
C_1 & = \begin{bmatrix} \frac{c_{N_0 + 1}}{\lambda_{N_0 +1}} & \ldots & \frac{c_{N + 1}}{\lambda_{N +1}} \end{bmatrix}
\end{align*}
in the case of the Neumann boundary measurement (\ref{eq: Neumann boundary measurement}), we infer that (\ref{eq: dynamics truncated model}) holds with $F$ defined by (\ref{eq: def matrix F}) with the updated version of $C_1$. The main interest of this re-scaling procedure is that the matrix $\Vert C_1 \Vert = O(1)$ as $N \rightarrow + \infty$, which will allow the application of Lemma~\ref{lem: useful lemma}; see~\cite{lhachemi2020finite} for details.

\begin{remark}
It can be seen that the pair $(A_0,C_0)$ is observable if and only if $c_n \neq 0$ for $1 \leq n \leq N_0$. In the case $c_n = \phi_n(0)$ with $\theta_1 \in (0,\pi/2]$, due to the boundary condition $\cos(\theta_1) \phi_n(0) - \sin(\theta_1) \phi_n'(0) = 0$, the identity $c_n = 0$ would imply that $\phi_n(0) = \phi_n'(0)$, giving the contradiction $\phi_n = 0$. A similar argument applies in the case $c_n = \phi_n'(0)$ with $\theta_1 \in [0,\pi/2)$. Hence, in both cases, the pair $(A_0,C_0)$ is observable.
\end{remark}

In the case of the Dirichlet boundary measurement (\ref{eq: Dirichlet boundary measurement}), the stability analysis can now be conducted similarly to the proof of Theorem~\ref{thm: main result 2} by modifying the estimate of $\zeta(t)$ by $\zeta(t)^2 \leq M_{1,\phi} \sum_{n \geq N+1} \lambda_n z_n(t)^2$ with $M_{1,\phi} = \sum_{n \geq N+1} \frac{\phi_n(0)^2}{\lambda_n} < +\infty$. We hence obtain the following result.

\begin{theorem}\label{thm: main result 3}
Let $\theta_1 \in(0,\pi/2]$, $\theta_2\in[0,\pi/2]$, $p \in \mathcal{C}^2([0,1])$ with $p > 0$, $\tilde{q} \in \mathcal{C}^0([0,1])$, and $\ell\in(\R_{>0})^m$. Let $q \in \mathcal{C}^0([0,1])$ and $q_c \in\R$ be such that (\ref{eq: decomposition of the reaction term}) holds with the additional constraint $q > 0$. Let $b_k \in L^2(0,1)$ for $1 \leq k \leq m$. Consider the reaction-diffusion system described by (\ref{eq: distributed measurement - RD system}) with left Dirichlet boundary measurement (\ref{eq: Dirichlet boundary measurement}). Let $N_0 \geq 1$ and $\delta > 0$ be given such that $- \lambda_n + q_c < -\delta < 0$ for all $n \geq N_0 +1$. Assume that for any $1 \leq n \leq N_0$, there exists $1 \leq k = k(n) \leq m$ such that $b_{n,k} \neq 0$. Let $K \in\R^{m \times N_0}$ and $L \in\R^{N_0}$ be such that $A_0 + B_0 K$ and $A_0 - L C_0$ are Hurwitz with eigenvalues that have a real part strictly less than $-\delta < 0$. For a given $N \geq N_0 +1$, assume that there exist a symmetric positive definite $P \in \R^{2N \times 2N}$, $\alpha > 1$, $\beta,\gamma,\mu,\kappa > 0$, a diagonal positive definite $T \in \R^{m \times m}$, and $C \in \R^{m \times N_0}$ such that 
\begin{equation}\label{eq: thm3 constraints}
\Theta_1(\kappa) \preceq 0, \quad \Theta_2 \succeq 0, \quad\Theta_3(\kappa) \leq 0
\end{equation}
where $\Theta_1(\kappa)$ and $\Theta_2$ are defined as in Theorem~\ref{thm: main result 1} while 
\begin{align*}
\Theta_3(\kappa) & = 2\gamma \left\{ - \left( 1 - \frac{1}{\alpha} \right) \lambda_{N+1} + q_c + \kappa \right\} + \beta M_{1,\phi} .
\end{align*}
Consider the block representation $P = (P_{i,j})_{1 \leq i,j \leq 4}$ with dimensions that are compatible with (\ref{eq: state-vector closed-loop truncated model}) and define 
\begin{align}
\mathcal{E}_3 & = \bigg\{
z \in D(\mathcal{A}) : \begin{bmatrix} \pi_{N_0}z \\ \pi_{N_0,N} \mathcal{A}^{1/2}z \end{bmatrix}^\top
\begin{bmatrix}
P_{2,2} & P_{2,4} \\ P_{4,2} & P_{4,4} 
\end{bmatrix}
\begin{bmatrix} \pi_{N_0}z \\ \pi_{N_0,N} \mathcal{A}^{1/2}z \end{bmatrix} \nonumber \\
& \hspace{2.4cm} + \gamma \Vert \mathcal{R}_N \mathcal{A}^{1/2}z \Vert_{L^2}^2
< \frac{1}{\mu}
\bigg\} . \label{eq: ellipsoid thm3}
\end{align}
Then, considering the closed-loop system composed of the plant (\ref{eq: distributed measurement - RD system}) with left Dirichlet boundary measurement (\ref{eq: Dirichlet boundary measurement}) and the control law (\ref{eq: control law}), there exists $M > 0$ such that for any initial condition $z_0 \in \mathcal{E}_3$ and with a zero initial condition of the observer (i.e., $\hat{z}_n(0)=0$ for all $1 \leq n \leq N$), the system trajectory satisfies $\Vert z(t,\cdot) \Vert_{H^1}^2 + \sum_{n=1}^{N} \hat{z}_n(t)^2 \leq M e^{-2 \kappa t} \Vert z_0 \Vert_{H^1}^2$ for all $t \geq 0$. Moreover, for any fixed $\kappa \in (0,\delta]$, the above constraints are always feasible for $N$ large enough.
\end{theorem}

In the case of the Neumann boundary measurement (\ref{eq: Neumann boundary measurement}), the stability analysis can also be conducted similarly to the proof of Theorem~\ref{thm: main result 2} by modifying the estimate of $\zeta(t)$ by $\zeta(t)^2 \leq M_{2,\phi}(\epsilon) \sum_{n \geq N+1} \lambda_n^{3/2+\epsilon} z_n(t)^2$ with $M_{2,\phi}(\epsilon) = \sum_{n \geq N+1} \frac{\phi_n'(0)^2}{\lambda_n^{3/2+\epsilon}} < +\infty$ for any given $\epsilon\in(0,1/2]$. We hence obtain the following result\footnote{In the case of Theorem~\ref{thm: main result 4}, the feasability of the constraints (\ref{eq: thm4 constraints}) for $N$ large enough is obtained by conducting the analysis as in the proof of Theorem~\ref{thm: main result 2} while selecting $\beta = \eta N^{1/8}$, $\gamma = \eta/N^{3/16}$, and $\epsilon = 1/8$.}.

\begin{theorem}\label{thm: main result 4}
Let $\theta_1 \in[0,\pi/2)$, $\theta_2\in[0,\pi/2]$, $p \in \mathcal{C}^2([0,1])$ with $p > 0$, $\tilde{q} \in \mathcal{C}^0([0,1])$, and $\ell\in(\R_{>0})^m$. Let $q \in \mathcal{C}^0([0,1])$ and $q_c \in\R$ be such that (\ref{eq: decomposition of the reaction term}) holds with the additional constraint $q > 0$. Let $b_k \in L^2(0,1)$ for $1 \leq k \leq m$. Consider the reaction-diffusion system described by (\ref{eq: distributed measurement - RD system}) with left Neumann boundary measurement (\ref{eq: Neumann boundary measurement}). Let $N_0 \geq 1$ and $\delta > 0$ be given such that $- \lambda_n + q_c < -\delta < 0$ for all $n \geq N_0 +1$. Assume that for any $1 \leq n \leq N_0$, there exists $1 \leq k = k(n) \leq m$ such that $b_{n,k} \neq 0$. Let $K \in\R^{m \times N_0}$ and $L \in\R^{N_0}$ be such that $A_0 + B_0 K$ and $A_0 - L C_0$ are Hurwitz with eigenvalues that have a real part strictly less than $-\delta < 0$. For a given $N \geq N_0 +1$, assume that there exist a symmetric positive definite $P \in \R^{2N \times 2N}$, $\alpha > 1$, $\beta,\gamma,\mu,\kappa > 0$, $\epsilon\in(0,1/2]$, a diagonal positive definite $T \in \R^{m \times m}$, and $C \in \R^{m \times N_0}$ such that 
\begin{equation}\label{eq: thm4 constraints}
\Theta_1(\kappa) \preceq 0, \quad \Theta_2 \succeq 0, \quad\Theta_3(\kappa) \leq 0 , \quad \Theta_4 \geq 0
\end{equation}
where $\Theta_1(\kappa)$ and $\Theta_2$ are defined as in Theorem~\ref{thm: main result 1} while 
\begin{align*}
\Theta_3(\kappa) & = 2\gamma \left\{ - \left( 1 - \frac{1}{\alpha} \right) \lambda_{N+1} + q_c + \kappa \right\} + \beta M_{2,\phi}(\epsilon) \lambda_{N+1}^{1/2+\epsilon} \\
\Theta_4 & = 2\gamma \left\{ 1 - \frac{1}{\alpha} \right\} - \dfrac{\beta M_{2,\phi}(\epsilon)}{\lambda_{N+1}^{1/2-\epsilon}} .
\end{align*}
Consider the block representation $P = (P_{i,j})_{1 \leq i,j \leq 4}$ with dimensions that are compatible with (\ref{eq: state-vector closed-loop truncated model}) and define
\begin{align}
\mathcal{E}_4 & = \bigg\{
z \in D(\mathcal{A}) : \begin{bmatrix} \pi_{N_0}z \\ \pi_{N_0,N} \mathcal{A}z \end{bmatrix}^\top
\begin{bmatrix}
P_{2,2} & P_{2,4} \\ P_{4,2} & P_{4,4} 
\end{bmatrix}
\begin{bmatrix} \pi_{N_0}z \\ \pi_{N_0,N} \mathcal{A}z \end{bmatrix} \nonumber \\
& \hspace{2.4cm} + \gamma \Vert \mathcal{R}_N \mathcal{A}^{1/2}z \Vert_{L^2}^2
< \frac{1}{\mu}
\bigg\} . \label{eq: ellipsoid thm4}
\end{align}
Then, considering the closed-loop system composed of the plant (\ref{eq: distributed measurement - RD system}) with left Neumann boundary measurement (\ref{eq: Neumann boundary measurement}) and the control law (\ref{eq: control law}), there exists $M > 0$ such that for any initial condition $z_0 \in \mathcal{E}_4$ and with a zero initial condition of the observer (i.e., $\hat{z}_n(0)=0$ for all $1 \leq n \leq N$), the system trajectory satisfies $\Vert z(t,\cdot) \Vert_{H^1}^2 + \sum_{n=1}^{N} \hat{z}_n(t)^2 \leq M e^{-2 \kappa t} \Vert z_0 \Vert_{H^1}^2$ for all $t \geq 0$. Moreover, for any fixed $\kappa \in (0,\delta]$, the above constraints are always feasible for $N$ large enough.
\end{theorem}

\begin{remark}
Even if Theorems~\ref{thm: main result 3} and~\ref{thm: main result 4} address the case of left boundary measurements, right boundary measurements can also be conducted similarly. This is also the case of in-domain measurements of the type $y(t) = z(t,\xi)$ and $y(t) = z_x(t,\xi)$ for some $\xi\in(0,1)$. In that case, one need to explicitly check that the pair $(A_0,C_0)$ is observable, which holds true if an only if $\phi_n(\xi) \neq 0$ for all $1 \leq n \leq N_0$ in the case of a Dirichlet measurement while $\phi_n'(\xi) \neq 0$ for all $1 \leq n \leq N_0$ in the case of a Neumann measurement.
\end{remark}

\section{Conclusion}\label{sec: conclusion}
This paper has studied the output feedback stabilization of a reaction-diffusion PDEs by means of bounded control inputs. The control strategy takes the form of a finite-dimensional controller. The reported stability analysis takes advantage of Lyapunov functionals along with a classical generalized sector condition that is commonly used in the analysis of finite-dimensional saturated systems. The obtained sets of constraints ensuring the stability of the closed-loop system take an explicit form and have been shown to be feasible when the order of the observer is large enough. An explicit subset of the domain of attraction of the closed-loop system has been derived.



\appendix

\section*{Technical lemma}

The statement of the following Lemma is borrowed from~\cite[Appendix]{lhachemi2020finite} and generalizes a result presented in~\cite{katz2020constructive}.

\begin{lemma}\label{lem: useful lemma}
Let $n,m,N \geq 1$, $M_{11} \in \R^{n \times n}$ and $M_{22} \in \R^{m \times m}$ Hurwitz, $M_{12} \in \R^{n \times m}$, $M_{14}^N \in\R^{n \times N}$, $M_{24}^N \in\R^{m \times N}$, $M_{31}^N \in\R^{N \times n}$, $M_{33}^N,M_{44}^N \in \R^{N \times N}$, and
\begin{equation*}
F^N = \begin{bmatrix}
M_{11} & M_{12} & 0 & M_{14}^N \\
0 & M_{22} & 0 & M_{24}^N \\
M_{31}^N & 0 & M_{33}^N & 0 \\
0 & 0 & 0 & M_{44}^N
\end{bmatrix} .
\end{equation*}
We assume that there exist constants $C_0 , \kappa_0 > 0$ such that $\Vert e^{M_{33}^N t} \Vert \leq C_0 e^{-\kappa_0 t}$ and $\Vert e^{M_{44}^N t} \Vert \leq C_0 e^{-\kappa_0 t}$ for all $t \geq 0$ and all $N \geq 1$. Moreover, we assume that there exists a constant $C_1 > 0$ such that $\Vert M_{14}^N \Vert \leq C_1$, $\Vert M_{24}^N \Vert \leq C_1$, and $\Vert M_{31}^N \Vert \leq C_1$ for all $N \geq 1$. Then there exists a constant $C_2 > 0$ such that, for any $N \geq 1$, there exists a symmetric matrix $P^N \in\R^{n+m+2N}$ with $P^N \succ 0$ such that $(F^N)^\top P^N + P^N F^N= - I$ and $\Vert P^N \Vert \leq C_2$.
\end{lemma}


\ifCLASSOPTIONcaptionsoff
  \newpage
\fi



\bibliographystyle{IEEEtranS}
\nocite{*}
\bibliography{IEEEabrv,mybibfile}

\end{document}